\documentclass[11pt,a4paper]{amsart}

\usepackage{amsmath}
\usepackage{amsfonts}
\usepackage{amssymb}
\usepackage{amsthm}
\usepackage{epsfig}
\usepackage{graphicx}
\usepackage{url}
\usepackage[usenames,dvipsnames]{xcolor}
\usepackage[colorlinks=true,linkcolor=Blue,citecolor=Green,urlcolor=black]{hyperref}
\usepackage{nicefrac}
\usepackage{aliascnt}		
\usepackage[alwaysadjust]{paralist}
\usepackage[T1]{fontenc}
\usepackage[utf8]{inputenc}

\usepackage{tikz}
\usetikzlibrary{shapes.geometric}

\usepackage{pdflscape}
\usepackage{caption}
\newcommand*\xbar[1]{%
   \hbox{%
     \vbox{%
       \hrule height 0.5pt 
       \kern0.5ex
       \hbox{%
         \kern-0.1em
         \ensuremath{#1}%
         \kern-0.1em
       }%
     }%
   }%
}

\newtheorem{theorem}{Theorem}[section]
\newaliascnt{lemma}{theorem}
\newaliascnt{corollary}{theorem}
\newaliascnt{definition}{theorem}
\newaliascnt{remark}{theorem}
\newaliascnt{proposition}{theorem}
\newaliascnt{conjecture}{theorem}
\newaliascnt{example}{theorem}
\newaliascnt{problem}{theorem}

\aliascntresetthe{lemma}

\newtheorem*{lemma*}{Lemma}

\aliascntresetthe{corollary}

\newtheorem*{corollary*}{Corollary}

\aliascntresetthe{definition}

\newtheorem*{definition*}{Definition}

\newtheorem{remark}[remark]{Remark}
\aliascntresetthe{remark}

\newtheorem*{remark*}{Remark}

\aliascntresetthe{proposition}

\newtheorem*{proposition*}{Proposition}

\aliascntresetthe{conjecture}

\newtheorem*{conjecture*}{Conjecture}

\aliascntresetthe{example}

\newtheorem*{example*}{Example}

\aliascntresetthe{problem}

\newtheorem*{problem*}{Problem}

\DeclareMathOperator{\inter}{int}

\def\L{\mathcal{L}}

\def\R{\mathbb{R}}

\def\Z{\mathbb{Z}}
\def\N{\mathbb{N}}

\DeclareMathOperator{\vol}{vol}

\def\e{\varepsilon}
\def\vp{\varphi}

\numberwithin{equation}{section}

\begin{document}

\title[Discrete lattice-periodic sets and Archimedean tilings]{A note on discrete lattice-periodic sets with an application to Archimedean tilings}

\author{Matthias Schymura}
\address{Institut f\"ur Mathematik, Freie Universit\"at Berlin, Arnimallee 2, 14195 Berlin, Germany}
\email{matthias.schymura@fu-berlin.de}

\author{Liping Yuan}
\address{College of Mathematics and Information Science, Hebei Normal University, 050016 Shijiazhuang, China}
\email{lpyuan@hebtu.edu.cn}

\thanks{The first author was supported by the Freie Universit\"at Berlin within the Excellence Initiative of the German Research Foundation.
The second author gratefully acknowledges financial support by National Natural Science   Foundation  of China (No.~11471095), Outstanding Youth Science Foundation  of Hebei Province (No.~2013205189),  Program for Excellent Talents  in University, Hebei Province (No.~GCC2014043) and Key Project of the Education   Department of Hebei Province (No.~zd2017043).}



\begin{abstract}
Cao \& Yuan obtained a Blichfeldt-type result for the vertex set of the edge-to-edge  tiling of the plane by regular hexagons.
Observing that every Archimedean tiling is the union of translates of a fixed lattice, we take a more general viewpoint and investigate basic questions for such point sets about the homogeneous and inhomogeneous problem in the Geometry of Numbers.
The Archimedean tilings nicely exemplify our results.
\end{abstract}

\maketitle

\section{Introduction}

We are motivated by a work of Cao \& Yuan~\cite{caoyuan2011a} who established a planar variant of a classical theorem of Blichfeldt in the Geometry of Numbers.
In fact, they were interested in replacing the integer lattice~$\Z^2$ by the vertex set of the Archimedean tiling of the plane by regular hexagons.
Their study suggests a conceptual investigation for the vertex set of \emph{any} Archimedean tiling, which we offer below.
The interest in extending classical results to these special sets was initiated by a paper of Ding \& Reay~\cite{dingreay1987the} from 1987, on variants of Pick's theorem.

The structural property that the Archimedean tilings all share, is that their vertex sets are finite unions of translates of a fixed lattice.
Taking this general viewpoint, we extend fundamental results in the Geometry of Numbers on the \emph{homogeneous} and \emph{inhomogeneous problem} to this setting.
The Blichfeldt-type theorems for the Archimedean tilings then drop out as a very special case.
It turns out that the elegant arguments that led to the results for lattices can often be adapted without much effort.
To this end, we focus only on the most fundamental questions in order to illustrate the flavor of the possible findings and the consequences for the Archimedean tilings.

Before we procced, we introduce some basic notation and terminology.
A \emph{lattice} is a discrete subgroup of $\R^n$, and in this paper it is always assumed to be of full-rank.
Each lattice $\Lambda$ can be defined by a basis matrix $B\in\R^{n\times n}$ as $\Lambda=B\Z^n$, which also allows to define the \emph{determinant} of $\Lambda$ to be $\det(\Lambda)=|\det(B)|$.
Given a (Lebesgue-)measurable set $X$, its volume, interior, and closure are denoted by $\vol(X)$, $\inter(X)$, and $\xbar{X}$, respectively.
A set $X$ is called \emph{$o$-symmetric} if $X=-X$.
For a vector $t\in\R^n$, and a second set $Z$, we define the translate $X+t$ of $X$ by $t$, and the Minkowski-sum $X+Z$ of~$X$ and~$Z$ as usual.
A standard textbook on the Geometry of Numbers is that of Gruber \& Lekkerkerker~\cite{gruberlekker1987geometry} to which we refer for extended background information.
We use basic notions and results on lattices without explicit mention, and instead suggest the reader to consult~\cite{gruberlekker1987geometry}, or the more recent book of Gruber~\cite{gruber2007convex}.

\section{Geometry of discrete lattice-periodic sets}
\label{sect_lat-periodic}

A fundamental principle concerning lattice points in measurable sets was introduced by Blichfeldt~\cite{blichfeldt1914anew} (cf.~\cite[\S 6]{gruberlekker1987geometry}):
Given a bounded measurable set $D\subseteq\R^n$ and a lattice $\Lambda$, there exists a translation $z\in\R^n$ such that
\begin{align}
\#\left((z+\xbar{D})\cap\Lambda\right)\geq\left\lfloor\frac{\vol(D)}{\det(\Lambda)}\right\rfloor+1.\label{eqn_blichfeldt_spec}
\end{align}

This result is at the heart of the Geometry of Numbers and spurred important developments in this discipline (cf.~\cite[Ch.~2]{gruberlekker1987geometry}).
It is thus no surprise that it has been repeatedly extended and generalized.
With our application to Archimedean tilings in mind, we subsequently formulate Blichfeldt's principle for finite unions of translates of a fixed lattice, and then derive consequences regarding a Minkowski-type theorem and basic covering statements.

A set $X\subseteq\R^n$ is called \emph{lattice-periodic} with period lattice $\Lambda\subseteq\R^n$ if $X+z=X$, for every $z\in\Lambda$.
A subset $X\subseteq\R^n$ is called \emph{discrete}, if every bounded subset of $\R^n$ contains only finitely many points of $X$.
Given a lattice $\Lambda$, the union
\begin{align}
\L(\Lambda,v_1,\ldots,v_k):=\bigcup_{i=1}^k(v_i+\Lambda),\label{eqn_discretelatperiodic}
\end{align}
of its translates by points $v_1,\ldots,v_k\in\R^n$ is clearly discrete and lattice-periodic.
Indeed, every discrete lattice-periodic set is of this form.

In order to see this, let $\{b_1,\ldots,b_n\}$ be a basis of $\Lambda$ and let
\[P=\left\{\sum_{i=1}^n m_ib_i : 0\leq m_i<1, i=1,\ldots,n\right\}\]
be the corresponding half-open fundamental cell.
Since $X$ is discrete there are finitely many points of~$X$ in $P$ which we call $v_1,\ldots,v_k$.
Since $X$ is moreover $\Lambda$-periodic, we get that $z+(P\cap X)=(P+z)\cap X$, for every $z\in\Lambda$ and hence $X=\L(\Lambda,v_1,\ldots,v_k)$.

In the sequel, we always assume that the $v_i$ are such that $v_i-v_j\notin\Lambda$, for every $1\leq i<j\leq k$, so that the representation~\eqref{eqn_discretelatperiodic} of $\L(\Lambda,v_1,\ldots,v_k)$ is minimal with respect to~$k$.

\subsection{The homogeneous problem for discrete lattice-periodic sets}

Although usually the version~\eqref{eqn_blichfeldt_spec} is cited as \emph{Blichfeldt's theorem}, his work~\cite{blichfeldt1914anew} actually contains a much more general version.
In modern language it reads as follows.

Let $P$ be a parallelepiped, that is, an affine image of a cube, and let $T\subseteq\R^n$ be a set of translation vectors such that $\bigcup_{t\in T}(P+t)=\R^n$ and $(\inter(P) +t)\cap(\inter(P)+s)=\emptyset$, for any distinct $s,t\in T$.

\begin{theorem}[Blichfeldt~\cite{blichfeldt1914anew}]\label{thm_blichfeldt_gen}
Let $D\subseteq\R^n$ be a bounded measurable set, and let $V\subseteq\R^n$ and $k\in\N$ be such that
\begin{enumerate}[i)]
 \item $\#(V\cap(P+t))=\#(V\cap(\inter(P)+t))=k$, for every $t\in T$, and
 \item $V\cap(P+t)$ is congruent to $V\cap(P+s)$, for every $s,t\in T$.
\end{enumerate}
Then, there exists a translation $z\in\R^n$ with
\[\#\left((z+\xbar{D})\cap V\right)\geq\left\lfloor\frac{\vol(D)\cdot k}{\vol(P)}\right\rfloor+1.\]
\end{theorem}

\noindent The version~\eqref{eqn_blichfeldt_spec} is the case of $V=\Lambda$ being a lattice, and where $P$ is a suitable translate of a fundamental cell of~$\Lambda$.
Discrete lattice-periodic sets fulfill the conditions in \autoref{thm_blichfeldt_gen}.
Since in this case the resulting bound turns out to be best possible, we include it as a separate statement.

\begin{theorem}\label{thm_blichfeldt_unions_of_translates}
Let $V=\L(\Lambda,v_1,\ldots,v_k)$, for some lattice $\Lambda\subseteq\R^n$ and $v_1,\ldots,v_k\in\R^n$.
Then, for every bounded measurable set $D\subseteq\R^n$, there exists a translation $z\in\R^n$ such that
\[\#\left((z+\xbar{D})\cap V\right)\geq\left\lfloor\frac{\vol(D)\cdot k}{\det(\Lambda)}\right\rfloor+1.\]
Moreover, the inequality is best possible for any $V$.
\end{theorem}
\begin{proof}
Let $P$ be a fundamental cell of the lattice $\Lambda$.
In particular, we have
\begin{align}
\left(x+\inter(P)\right)\cap\left(y+\inter(P)\right)=\emptyset\quad\text{for every}\quad x,y\in\Lambda, x\neq y,\label{eqn_fundCell}
\end{align}
and $\bigcup_{x\in\Lambda}(x+P)=\R^n$.
Since $V$ is a discrete set, there is a translate~$P'$ of~$P$ such that $\#(V\cap P')=\#(V\cap\inter(P'))$.
By the periodicity of~$V$, we clearly get that $\#(V\cap(P'+x))=\#(V\cap(\inter(P')+x))=k$, and $V\cap(x+P')=x+(V\cap P')$, for every $x\in\Lambda$.
Therefore, in view of \autoref{thm_blichfeldt_gen} there exists a $z\in\R^n$ such that
\[\#\left((z+\xbar{D})\cap V\right)\geq\left\lfloor\frac{\vol(D)\cdot k}{\vol(P')}\right\rfloor+1=\left\lfloor\frac{\vol(D)\cdot k}{\det(\Lambda)}\right\rfloor+1.\]
We finish the proof by showing that the inequality is best possible.
For $\e>0$, we write $P_\e=(1-\e)P$ and observe that for $D=P_\e$ and $\e$ small enough the just proven lower bound evaluates to $\left\lfloor\frac{\vol(P_\e)\cdot k}{\det(\Lambda)}\right\rfloor+1=\left\lfloor(1-\e)^n k\right\rfloor+1=k$.
Now, assume that there exists a translation $z\in\R^n$ such that $z+P_\e$ contains at least $k+1$ points of $V$, say $z_0,z_1,\ldots,z_k$.
The pigeonhole principle provides us with a pair of indices $0\leq i< j\leq k$ such that $z_i+\Lambda=z_j+\Lambda$, that is, $z_i-z_j\in(\Lambda-\Lambda)\setminus\{0\}$.
As a consequence we find
\[z_i-z_j\in(P_\e-P_\e)\cap(\Lambda-\Lambda)\setminus\{0\}\subseteq\inter(P-P)\cap(\Lambda-\Lambda)\setminus\{0\},\]
which is in contradiction with \eqref{eqn_fundCell}.
\end{proof}

One of the many applications of Blichfeldt's principle is an alternative proof of Minkowski's fundamental theorem.
This is an instance of the \emph{homogeneous problem} in the Geometry of Numbers: What are conditions for a convex set to contain a non-zero lattice point?

Minkowski~\cite{minkowski1896geometrie} (cf.~\cite[\S 5]{gruberlekker1987geometry}) showed that given a closed $o$-symmetric convex set $K\subseteq\R^n$ and a lattice $\Lambda\subseteq\R^n$ such that $\vol(K)\geq2^n\det(\Lambda)$, there will always be a non-zero point $w\in K\cap\Lambda$.
Using ~\autoref{thm_blichfeldt_unions_of_translates}, this extends to discrete lattice-periodic sets as follows.

\begin{theorem}\label{thm_minkowski_type}
Let $V=\L(\Lambda,v_1,\ldots,v_k)$ be a discrete lattice-periodic set, and let $K\subseteq\R^n$ be a closed $o$-symmetric convex set with $\vol(K)\geq 2^n\ell\frac{\det(\Lambda)}{k}$.
Then,
\begin{enumerate}[i)]
 \item $\#\left(K\cap(V-V)\right)\geq2\ell+1$, and
 \item if $V-V=V\cup(-V)$, then $\#\left(K\cap V\right)\geq\ell+1$.
\end{enumerate}
\end{theorem}
\begin{proof}
The proof goes along the exact same lines as the one given in~\cite[\S 7.2, Thm.~1]{gruberlekker1987geometry}.
By assumption, $\vol(\frac12K)\geq\ell\frac{\det(\Lambda)}{k}$, and hence by \autoref{thm_blichfeldt_unions_of_translates}, there exists some $z\in\R^n$ and some points $v_1,\ldots,v_{\ell+1}\in V$, such that $z+v_i\in\frac12 K$, for $i=1,\ldots,\ell+1$.
Putting $u_i=v_{i+1}-v_1$, for $i=1,\ldots,\ell$, we obtain that $\pm u_i\neq 0$ are $\ell$ pairs of points that are contained in $\frac12K-\frac12K$, which equals $K$ due to its convexity.
Together with the origin, this makes $2\ell+1$ points contained in $K\cap(V-V)$ and thus proves part~\romannumeral1).

For part~\romannumeral2), we first observe that if $V-V=V\cup(-V)$, we necessarily have $0\in V$.
Moreover, at least one of the previously constructed points~$u_i$ and $-u_i$ is contained in $V$, for every $i=1,\ldots,\ell$, and hence $\#\left(K\cap V\right)\geq\ell+1$.
\end{proof}

\begin{remark}
Guih{\'e}neuf \& Joly~\cite{guiheneufjoly2017aminkowski} recently proved a Minkowski-type theorem for general quasicrystals, and they note that this is actually a result on the difference set of the point set in question.
In this spirit, the first part of \autoref{thm_minkowski_type} seems to be the more natural.
If $V$ is a lattice, then clearly $V-V=V$, so part~\romannumeral1) extends Minkowski's theorem.
\end{remark}

\subsection{The inhomogeneous problem for discrete lattice-periodic sets}

In contrast to the homogeneous problem the \emph{inhomogeneous problem} asks for conditions on a given set such that every of its translates contains a lattice point.
Note that this is a covering condition: In fact, every translate of $D\subseteq\R^n$ contains a point of a lattice $\Lambda$ if and only if $D+\Lambda=\R^n$.

The fundamental fact is now that every measurable set $D$ with $D+\Lambda=\R^n$ necessarily has a volume at least as large as the determinant of~$\Lambda$ (see~\cite[\S 13.5]{gruberlekker1987geometry}).
This extends almost verbatim to every discrete lattice-periodic set.

\begin{theorem}\label{thm_inhom_vol_lat-per}
Let $V=\L(\Lambda,v_1,\ldots,v_k)$ be a discrete lattice-periodic set and let $D\subseteq\R^n$ be bounded and measurable.
If $D+V=\R^n$, then $\vol(D)\geq\frac{\det(\Lambda)}{k}$.
\end{theorem}
\begin{proof}
The proof follows the standard arguments for the lattice case:
Let~$P$ be a fundamental cell of the lattice~$\Lambda$.
Since all the upcoming sums are finite, we have
\begin{align*}
 k\vol(D)&=\sum_{i=1}^k\sum_{w\in\Lambda}\vol((w+P)\cap(D-v_i))\\
 &=\sum_{i=1}^k\sum_{w\in\Lambda}\vol((v_i+w+P)\cap D)\geq\sum_{v\in V}\vol((v+P)\cap D)\\
 &=\sum_{v\in V}\vol(P\cap (D-v))\geq\vol(P)=\det(\Lambda),
\end{align*}
where the last inequality holds due to the assumption $D+V=\R^n$.
\end{proof}

Another instance of the inhomogeneous problem can be derived from the assumption that we are given two domains $D_1,D_2\subseteq\R^n$ with the property that no translate contains many points of $V$, that is, neither of them satisfies the consequence of \autoref{thm_blichfeldt_unions_of_translates}.
Then, if not both $D_1$ and $D_2$ have small volume, then the difference set $D_1-D_2$ contains points of $V-V$ in every position (see~\cite[\S 13, Thm.~5]{gruberlekker1987geometry} for the case of lattices).

\begin{theorem}\label{thm_latptcovthm}
Let $V=\L(\Lambda,v_1,\ldots,v_k)$ be a discrete lattice-periodic set, let $\ell\in\N$, and let $D_1,D_2\subseteq\R^n$ be bounded and measurable such that
\begin{enumerate}[i)]
 \item $\vol(D_1)+\vol(D_2) > \frac{\det(\Lambda)}{k}\ell$, and
 \item $\#\left((z+D_i)\cap V\right)\leq \ell$, for $i=1,2$ and every $z\in\R^n$.
\end{enumerate}
Then, every translate of $D_1-D_2$ contains at least $\ell$ points of $V-V$.
\end{theorem}
\begin{proof}
We follow closely the arguments in the proof of~\cite[\S 13, Thm.~5]{gruberlekker1987geometry}.
First of all, observe that the assumptions are invariant under translations of $D_1$, so it suffices to show that $D_1-D_2$ contains at least $\ell$ points of $V-V$.

To this end, let $\chi_i$ be the characteristic function of $D_i$, for $i=1,2$, that is, $\chi_i(x)$ equals $1$ if $x\in D_i$, and $0$ otherwise.
Let
\[\vp_i(x)=\sum_{u\in V}\chi_i(x+u)=\#\left((D_i-x)\cap V\right),\quad\text{for}\quad x\in\R^n.\]
By \romannumeral2), we have $\vp_i(x)\leq\ell$, for every $x\in\R^n$ and $i=1,2$.
Let $P$ be a fundamental cell of $\Lambda$.
Then,
\begin{align*}
\int_P\vp_i(x)dx&=\int_P\sum_{u\in V}\chi_i(x+u)dx=\int_P\sum_{w\in\Lambda}\sum_{j=1}^k\chi_i(x+v_j+w)dx\\
&=\sum_{j=1}^k\sum_{w\in\Lambda}\int_P\chi_i(x+v_j+w)dx\\
&=\sum_{j=1}^k\sum_{w\in\Lambda}\vol\left(P\cap(D_i-v_j-w)\right)\\
&=\sum_{j=1}^k\vol(D_i-v_j)=k\vol(D_i).
\end{align*}
Therefore, using \romannumeral1), we get
\[\int_P\left(\vp_1(x)+\vp_2(x)\right)dx=k\left(\vol(D_1)+\vol(D_2)\right)>\ell\det(\Lambda).\]
Since $\vp_i(x)\in\Z$ and $\vol(P)=\det(\Lambda)$, this implies that there is some $x\in P$ such that $\vp_1(x)+\vp_2(x)\geq\ell+1$.
Together with $\vp_i(x)\leq\ell$, we get that $1\leq\vp_i(x)\leq\ell$, and hence there are $\ell+1$ points $u_1,\ldots,u_{\ell+1}\in V$ and some $1\leq p\leq \ell$, such that
\[x+u_i\in D_1,\text{ for }1\leq i\leq p,\quad\text{and}\quad x+u_j\in D_2,\text{ for }p+1\leq j\leq\ell+1.\]
Putting $U=\{u_1,\ldots,u_p\}$ and $U'=\{u_{p+1},\ldots,u_{\ell+1}\}$, we obtain $U-U'\subseteq(D_1-D_2)\cap(V-V)$.
Finally, using the easy estimate $\#(U-U')\geq\# U+\# U'-1=\ell$, we get the desired $\ell$ points of $V-V$ that are contained in the difference set $D_1-D_2$.
\end{proof}

A consequence of \autoref{thm_latptcovthm} in the case $\ell=1$ is that if $D_1$ and $D_2$ are such that $\vol(D_1)+\vol(D_2)>\frac{\det(\Lambda)}{k}$ and every of their translates contain at most one point of~$V$, then $(D_1-D_2)+(V-V)$ is a covering of $\R^n$.

\section{Application to vertex sets of Archimedean tilings}

An \emph{Archimedean tiling} is a tiling of the plane with the following properties:
\begin{enumerate}[(1)]
 \item all tiles are regular polygons;
 \item the tiling is \emph{edge-to-edge}, that is, every two tiles intersect in a common vertex or edge, or not at all;
 \item the tiling is \emph{vertex-transitive}, that is, for every two vertices there is a symmetry of the tiling that takes one to the other.
\end{enumerate}

By the third of these properties, one can identify an Archimedean tiling by the ordered sequence of polygons meeting at a fixed vertex.
More precisely, we write $(n_1.n_2\cdots n_t)$ for the tiling where at each vertex, an $n_1$-gon is next to an $n_2$-gon which is next to an $n_3$-gon, and so forth.
We also group repeated occurences of the same $n_i$-gon into exponents, so for example, the tiling consisting only of squares is denoted by $(4^4)=(4.4.4.4)$.

It is well-known that there are exactly $11$ Archimedean tilings.
Three tilings with one type of tile: $(4^4),(3^6),(6^3)$, six tilings with two types of tiles: $(3^3.4^2),(3.6.3.6),(3^2.4.3.4),(4.8^2),(3^4.6),(3.12^2)$, and two tilings with three types of tiles: $(3.4.6.4),(4.6.12)$.
We refer to~\cite{gruenbaumshephard1977tilings} for a proof of this classification, more background information, and illustrations for each of the tilings.
Here, in \autoref{fig_488tiling}, we exemplify the Archimedean tilings by an illustration of $(4.8^2)$.

\begin{figure}[ht]
 \includegraphics[scale=.4]{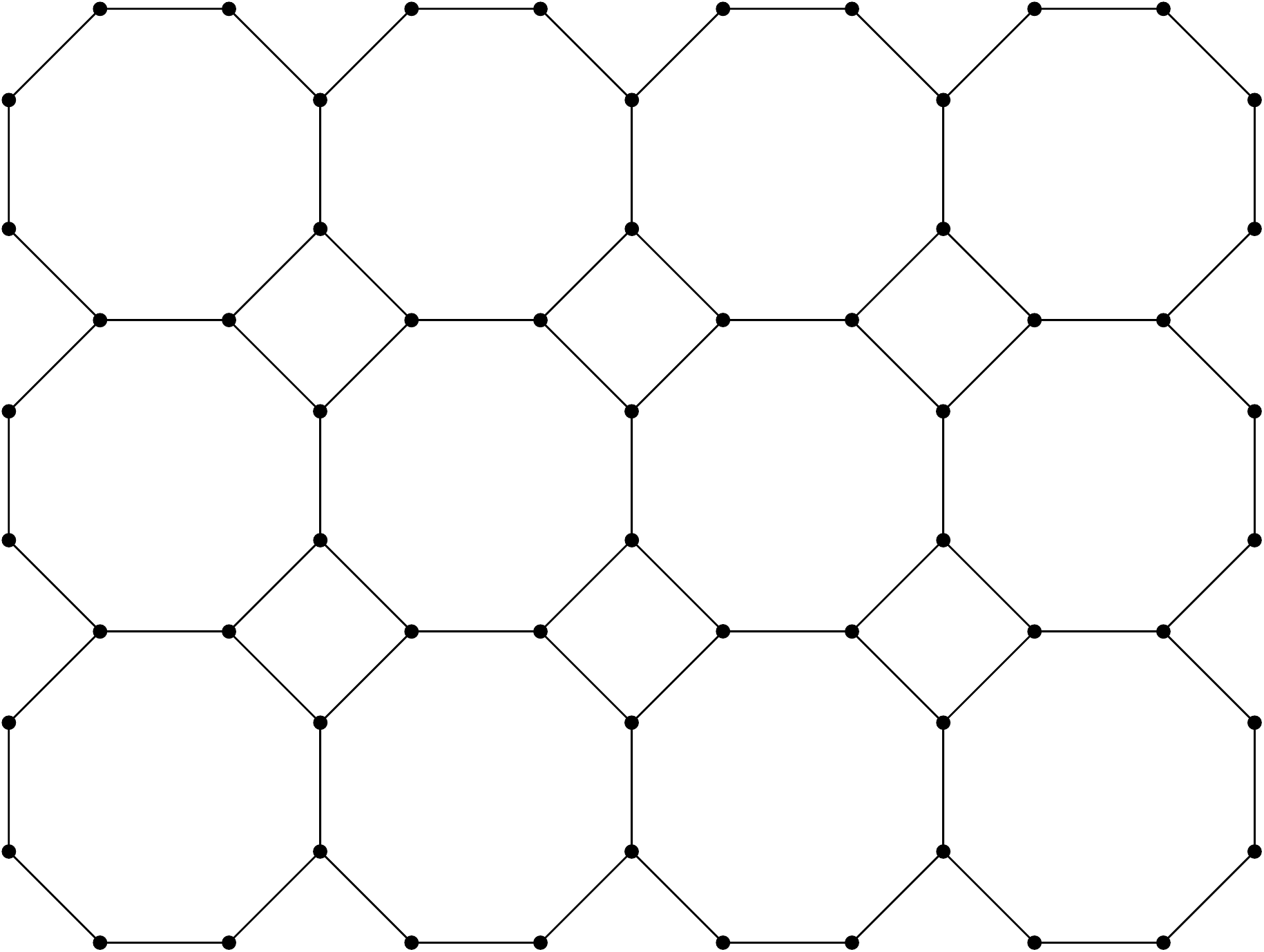}
\caption{The $(4.8^2)$ tiling.}
\label{fig_488tiling}
\end{figure}

As discussed in the introduction, our starting point for this work is the Blichfeldt-type result by Cao \& Yuan~\cite{caoyuan2011a} for what they call \emph{$H$-points}, that is, vertices of the tiling $(6^3)$.
The set of vertices of an Archimedean tiling is in general not a lattice.
Thus, a study of Blichfeldt's principle for the lattice~$\Lambda$ in~\eqref{eqn_blichfeldt_spec} replaced by these discrete sets is of interest.
However, we see below that the vertex sets of all $11$ tilings are special cases of discrete lattice-periodic sets, and hence they obey the results derived in \autoref{sect_lat-periodic}.

\renewcommand{\arraystretch}{1.5}
\begin{landscape}
\begin{table}[ht]
\small
\begin{tabular}[h]{|c|c|c|}
\hline
tiling & $B$ & $(v_1,\ldots,v_k)$ \\ \hline\hline

$(4^4)$ & $\Bigg(\begin{array}{cc}1 & 0\\ 0 & 1\end{array}\Bigg)$ & $\Bigg(\begin{array}{c}0\\ 0\end{array}\Bigg)$ \\ \hline

$(3^6)$ & $\Bigg(\begin{array}{cc}1& \tfrac12\\ 0& \tfrac{\sqrt{3}}{2}\end{array}\Bigg)$ & $\Bigg(\begin{array}{c}0\\ 0\end{array}\Bigg)$ \\ \hline

$(6^3)$ & $\Bigg(\begin{array}{cc}\frac{3}{2}& 0\\ \frac{\sqrt{3}}2& \sqrt{3}\end{array}\Bigg)$ & $\Bigg(\begin{array}{cc}0 & 1\\ 0& 0\end{array}\Bigg)$ \\ \hline \hline

$(3^3.4^2)$ & $\Bigg(\begin{array}{cc}1+\frac{\sqrt{3}}{2}& 0\\ \frac12&1\end{array}\Bigg)$ & $\Bigg(\begin{array}{cc}0 & 1\\ 0& 0\end{array}\Bigg)$ \\ \hline

$(3.6.3.6)$ & $\Bigg(\begin{array}{cc}1& 0\\ \sqrt{3}&2\sqrt{3}\end{array}\Bigg)$ & $\Bigg(\begin{array}{ccc}0 & 1 & \tfrac32\\ 0 & 0& \tfrac{\sqrt{3}}2\end{array}\Bigg)$ \\ \hline

$(3^2.4.3.4)$ & $\Bigg(\begin{array}{cc}1+\sqrt{3}& \tfrac12+\tfrac{\sqrt{3}}{2}\\ 0& \tfrac12+\tfrac{\sqrt{3}}{2}\end{array}\Bigg)$ & $\Bigg(\begin{array}{cccc}0 & \sqrt{3} & \tfrac{\sqrt{3}}2 & \tfrac{\sqrt{3}}2\\ 0 & 0 & \tfrac12 & -\tfrac12\end{array}\Bigg)$ \\ \hline

$(4.8^2)$ & $\Bigg(\begin{array}{cc} 1+\sqrt{2} & 0\\ 0 & 1+\sqrt{2}\end{array}\Bigg)$ & $\Bigg(\begin{array}{cccc}0 & 1 & 1+\tfrac{\sqrt{2}}2 & 1+\tfrac{\sqrt{2}}2\\ 0 & 0 & \tfrac{\sqrt{2}}2 & -\tfrac{\sqrt{2}}2\end{array}\Bigg)$ \\ \hline

$(3^4.6)$ & $\Bigg(\begin{array}{cc}\frac12& \sqrt{3}\\ \frac{3}{2}\sqrt{3} & -\sqrt{3}\end{array}\Bigg)$ & $\Bigg(\begin{array}{cccccc}0 & 1 & 0 & 1 & \tfrac32 & -\tfrac12 \\ 0 & 0 & \sqrt{3} & \sqrt{3} & \tfrac{\sqrt{3}}2 & \tfrac{\sqrt{3}}2 \end{array}\Bigg)$ \\ \hline

$(3.12^2)$ & $\Bigg(\begin{array}{cc}\frac32+\sqrt{3}& 0\\ 1+\frac{\sqrt{3}}{2}&2+\sqrt{3}\end{array}\Bigg)$ & {\small $\Bigg(\begin{array}{cccccc}0 & 1 & 1+\tfrac{\sqrt{3}}2 & 1+\tfrac{\sqrt{3}}2 & \tfrac32+\tfrac{\sqrt{3}}2 & \tfrac32+\tfrac{\sqrt{3}}2 \\ 0 & 0 & \tfrac12 & -\tfrac12 & \tfrac12+\tfrac{\sqrt{3}}2 & -\tfrac12-\tfrac{\sqrt{3}}2 \end{array}\Bigg)$} \\ \hline \hline

$(3.4.6.4)$ & $\Bigg(\begin{array}{cc}\frac32+\frac{\sqrt{3}}{2}& 0\\ -\frac12-\frac{\sqrt{3}}{2}&1+\sqrt{3}\end{array}\Bigg)$ & $\Bigg(\begin{array}{cccccc}0 & 1 & 0 & 1 & \tfrac32 & -\tfrac12 \\ 0 & 0 & \sqrt{3} & \sqrt{3} & \tfrac{\sqrt{3}}2 & \tfrac{\sqrt{3}}2 \end{array}\Bigg)$ \\ \hline

$(4.6.12)$ & $\Bigg(\begin{array}{cc}3+\sqrt{3}& \tfrac32+\tfrac{\sqrt{3}}{2}\\ 0 & \tfrac32+\tfrac{3}{2}\sqrt{3}\end{array}\Bigg)$ & {\small $\Bigg(\begin{array}{cccccccccccc}0 & 1 & 0 & 1 & 1+\tfrac{\sqrt{3}}2 & -\tfrac{\sqrt{3}}2 & 1+\tfrac{\sqrt{3}}2 & -\tfrac{\sqrt{3}}2 & \tfrac{3+\sqrt{3}}2 & -\tfrac{1+\sqrt{3}}2 & \tfrac{3+\sqrt{3}}2 & -\tfrac{1+\sqrt{3}}2\\ 0 & 0 & 2+\sqrt{3} & 2+\sqrt{3} & \tfrac12 & \tfrac12 & \tfrac32+\sqrt{3} & \tfrac32+\sqrt{3} & \tfrac{1+\sqrt{3}}2 & \tfrac{1+\sqrt{3}}2 & \tfrac{3+\sqrt{3}}2 & \tfrac{3+\sqrt{3}}2 \end{array}\Bigg)$} \\ \hline
\end{tabular}
\caption{Vertex sets of the Archimedean tilings as discrete lattice-periodic sets.}
\label{fig_AT_representations}
\end{table}
\end{landscape}
\renewcommand{\arraystretch}{1.0}

In \autoref{fig_AT_representations}, one finds a basis matrix $B$ for the period lattice $\Lambda=B\Z^2$ for each tiling, together with translation vectors $v_1,\ldots,v_k$ organized in a $(2\times k)$-matrix.
We assume that the involved polygons have unit edge-length, that one edge of the largest involved polygon is horizontal, and that one of its vertices is at the origin.

\renewcommand{\arraystretch}{1.5}
\begin{table}[ht]
\begin{tabular}[h]{|c|c|c|c|}
\hline
tiling & $k$ & $\det(\Lambda)$ & lower bound in \autoref{thm_blichfeldt_unions_of_translates} \\ \hline\hline

$(4^4)$ & $1$ & $1$ & $\left\lfloor\vol(D)\right\rfloor+1$ \\ \hline

$(3^6)$ & $1$ & $\tfrac12\sqrt{3}$ & $\left\lfloor\tfrac23\sqrt{3}\vol(D)\right\rfloor+1$ \\ \hline

$(6^3)$ & $2$ & $\tfrac32\sqrt{3}$ & $\left\lfloor\tfrac49\sqrt{3}\vol(D)\right\rfloor+1$ \\ \hline \hline

$(3^3.4^2)$ & $2$ & $1+\tfrac{1}{2}\sqrt{3}$ & $\left\lfloor(8-4\sqrt{3})\vol(D)\right\rfloor+1$ \\ \hline

$(3.6.3.6)$ & $3$ & $2\sqrt{3}$ & $\left\lfloor\tfrac{1}{2}\sqrt{3}\vol(D)\right\rfloor+1$ \\ \hline

$(3^2.4.3.4)$ & $4$ & $2+\sqrt{3}$ & $\left\lfloor(8-4\sqrt{3})\vol(D)\right\rfloor+1$ \\ \hline

$(4.8^2)$ & $4$ & $3+2\sqrt{2}$ & $\left\lfloor(12-8\sqrt{2})\vol(D)\right\rfloor+1$ \\ \hline

$(3^4.6)$ & $6$ & $\tfrac72\sqrt{3}$ & $\left\lfloor\tfrac47\sqrt{3}\vol(D)\right\rfloor+1$ \\ \hline

$(3.12^2)$ & $6$ & $6+\tfrac{7}{2}\sqrt{3}$ & $\left\lfloor(28\sqrt{3}-48)\vol(D)\right\rfloor+1$ \\ \hline \hline

$(3.4.6.4)$ & $6$ & $3+2\sqrt{3}$ & $\left\lfloor(4\sqrt{3}-6)\vol(D)\right\rfloor+1$ \\ \hline

$(4.6.12)$ & $12$ & $9+6\sqrt{3}$ & $\left\lfloor(\tfrac83\sqrt{3}-4)\vol(D)\right\rfloor+1$ \\ \hline
\end{tabular}
\caption{Blichfeldt-type theorems for the Archimedean tilings.}
\label{fig_BTthms}
\end{table}
\renewcommand{\arraystretch}{1.0}

Based on these explicit representations, we can now apply \autoref{thm_blichfeldt_unions_of_translates} and get an optimal Blichfeldt-type theorem for each Archimedean tiling.
The results are gathered in \autoref{fig_BTthms}.
Of course, this includes the result of~\cite{caoyuan2011a}, which looks a bit different though, because they assume the hexagons in the $(6^3)$ tiling to have area one, instead of edge-length one.

Cao \& Yuan~\cite{caoyuan2011a} also state a Minkowski-type result for the set of vertices~$H$ of the $(6^3)$ tiling, which reads as follows (assuming edge-length equal to one).

\begin{theorem}[\cite{caoyuan2011a}]\label{thm_mink_h-points}
Let $D\subseteq\R^2$ be a closed convex set that is centrally symmetric about a point of~$H$ and with area at least $2\sqrt{3}$.
Then, $D$ contains at least one other point of~$H$.
\end{theorem}

This is a statement in the spirit of \autoref{thm_minkowski_type}~\romannumeral2), so let us check which tilings have the property that their vertex set $V$ satisfies $V-V=V\cup(-V)$.
For a general discrete lattice-periodic set $V=\L(\Lambda,v_1,\ldots,v_k)$ we have
\[V-V=\bigcup_{i,j=1}^k(v_i-v_j+\Lambda)\quad\text{and}\quad V\cup(-V)=\bigcup_{i=1}^k(\pm v_i+\Lambda).\]
Hence, $V-V=V\cup(-V)$ if and only if, for all $i,j\in\{1,\ldots,k\}$, there is some $t\in\{1,\ldots,k\}$ such that $v_i-v_j\in \pm v_t+\Lambda$.

This simple criterion is satisfied only by the tilings $(4^4)$, $(3^6)$, $(6^3)$, $(3^3.4^2)$, and $(3^4.6)$.
Moreover, for the $(6^3)$ tiling, \autoref{thm_minkowski_type} guarantees at least two points in $D\cap H$ under the condition that $\vol(D)\geq3\sqrt{3}$, which is clearly weaker than \autoref{thm_mink_h-points}.
Hence, in order to get optimal Minkowski-type results, a particular study for each Archimedean tiling seems to be necessary.

\begin{figure}[ht]
 \includegraphics[scale=.4]{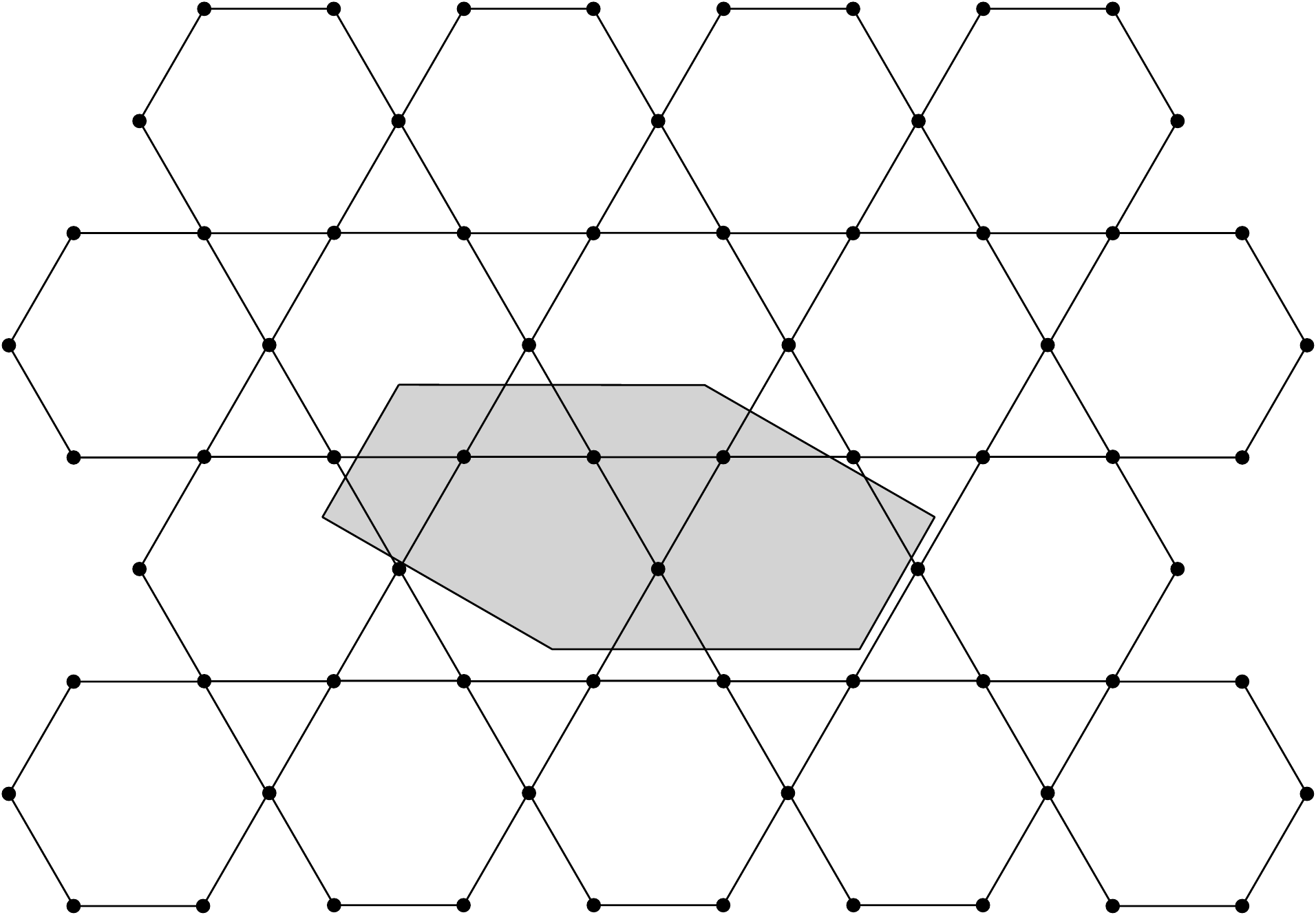}
\caption{Every translate of the shaded hexagon contains at least three vertices of $(3.6.3.6)$.}
\label{fig_3636_latptcov}
\end{figure}

Finally, we discuss an example of an Archimedean tiling that illustrates \autoref{thm_latptcovthm}.
The shaded hexagon in \autoref{fig_3636_latptcov} has the property that each of its translates contains at least three vertices of the $(3.6.3.6)$ tiling.

Again, this can be explained generally on a discrete lattice-periodic set $V=\L(\Lambda,v_1,\ldots,v_k)$.
Let $P$ and $Q$ be two fundamental cells of $\Lambda$, and for $\e>0$, let $D_1=(1-\e)P$ and $D_2=(1-\e)Q$.
Since $\vol(P)=\vol(Q)=\det(\Lambda)$, for $\e$ small enough and $\ell=k$, we get that $\vol(D_1)+\vol(D_2)>\det(\Lambda)\frac{\ell}{k}$ and $\#\left((z+D_i)\cap V\right)\leq \ell=k$, for $i=1,2$ and every $z\in\R^n$.
The latter follows from the arguments in the proof of \autoref{thm_blichfeldt_unions_of_translates}.
Therefore, the assumptions of \autoref{thm_latptcovthm} are satisfied, and as a result every translate of $D_1-D_2=(1-\e)(P-Q)$ contains at least $k$ points of $V-V$.

Now, we specialize to $V$ being the vertex set of $(3.6.3.6)$.
This tiling has the property that in fact we have $V-V=V\cup(-V)=V$, and from \autoref{fig_AT_representations} we know that $k=3$.
The difference set of two rectangles is either a rectangle, a hexagon, or an octagon.

\bibliographystyle{amsplain}
\bibliography{mybib}

\end{document}